\documentclass[11pt,A4paper]{article}

\usepackage[left=32mm,right=32mm,top=30mm,bottom=32mm]{geometry}
\usepackage{labelfig}
\usepackage{epsfig}
\usepackage{epstopdf}
\usepackage{psfrag}
\usepackage{xfrac}

\usepackage[all]{xy}

\usepackage{color}

\usepackage{amsthm,amsmath,amssymb}
\usepackage{booktabs}
\usepackage{txfonts}
\usepackage{mathpazo}
\usepackage{microtype}
\usepackage{overpic}
\usepackage{bm}
\usepackage{sectsty}
\usepackage[
	pdftitle={PDFTitle},
	pdfauthor={Hugo Parlier},
	ocgcolorlinks,
	linkcolor=linkred,
	citecolor=linkred,
	urlcolor=linkblue]
{hyperref}

\definecolor{linkred}{rgb}{0.48,0.1,0.05}
\definecolor{linkblue}{RGB}{16, 78, 139}

\usepackage[hang,flushmargin]{footmisc}
\usepackage{enumitem}
\usepackage{titlesec}
	\titlespacing{\section}{0pt}{12pt}{0pt}
	\titlespacing{\subsection}{0pt}{6pt}{0pt}
	
\titlelabel{\thetitle.\quad}

\makeatletter 

\long\def\@footnotetext#1{%
\H@@footnotetext{%
\ifHy@nesting 
\hyper@@anchor{\@currentHref}{#1}%
\else 
\Hy@raisedlink{\hyper@@anchor{\@currentHref}{\relax}}#1%
\fi 
}}

\def\@footnotemark{%
\leavevmode 
\ifhmode\edef\@x@sf{\the\spacefactor}\nobreak\fi 
\H@refstepcounter{Hfootnote}%
\hyper@makecurrent{Hfootnote}%
\hyper@linkstart{link}{\@currentHref}%
\@makefnmark 
\hyper@linkend 
\ifhmode\spacefactor\@x@sf\fi 
\relax 
}%

\ifFN@multiplefootnote%
\renewcommand*\@footnotemark{%
\leavevmode 
\ifhmode 
\edef\@x@sf{\the\spacefactor}%
\FN@mf@check 
\nobreak 
\fi 
\H@refstepcounter{Hfootnote}%
\hyper@makecurrent{Hfootnote}%
\hyper@linkstart{link}{\@currentHref}%
\@makefnmark 
\hyper@linkend 
\ifFN@pp@towrite 
\FN@pp@writetemp 
\FN@pp@towritefalse 
\fi 
\FN@mf@prepare 
\ifhmode\spacefactor\@x@sf\fi 
\relax%
}%
\fi 

\makeatother 

\theoremstyle{plain}
\newtheorem{theorem}{Theorem}[section]

\newtheorem{lemma}[theorem]{Lemma}

\theoremstyle{definition}

\newtheorem{remark}[theorem]{Remark}

\usepackage{mathtools}

\newcommand{\diam}{\ensuremath{\mathrm{diam}}}

\newcommand{\MF}{\mathcal M \mathcal F}
\newcommand{\MFS}{\mathcal M \mathcal F^s}

\sectionfont{\large \sc}
\subsectionfont{\normalsize}

\long\def\symbolfootnote[#1]#2{\begingroup%
\def\thefootnote{\fnsymbol{footnote}}\footnote[#1]{#2}\endgroup}

\def\blfootnote{\xdef\@thefnmark{}\@footnotetext}

\begin{document}

{\Large \bfseries \sc Simultaneous flips on triangulated surfaces}

{\bfseries Valentina Disarlo, 
Hugo Parlier\symbolfootnote[1]{\normalsize Research supported by Swiss National Science Foundation grants numbers PP00P2\textunderscore 15302 and PP00P2\textunderscore 128557.
}}

{\em Abstract.} 
We investigate a type of distance between triangulations on finite type surfaces where one moves between triangulations by performing simultaneous flips. We consider triangulations up to homeomorphism and our main results are upper bounds on distance between triangulations that only depend on the topology of the surface.
\vspace{1cm}

\section{Introduction}\label{sec:intro}

The general theme of defining and measuring distances between triangulations on surfaces plays a role in the study of geometric topology, the geometric group theory perspective of mapping class groups and in combinatorial geometry. 

A usual measure of distance is to consider flip distance where one measures distance by considering the number of flip moves necessary to go from one triangulation to another. Associated to this measure are {\it flip graphs} where vertices are triangulations and there is an edge between vertices if the corresponding triangulations differ by a flip. These graphs appear in a number of contexts, most famously perhaps when the underlying surface is a polygon and in this case the flip graph is the 1-skeleton of a polytope (the associahedron) \cite{Stasheff, Tamari}; these graphs are finite and their diameters are now completely known \cite{Pournin,STT1}. In general, provided the surface has enough topology, flip graphs aren't  finite and are combinatorial models for homeomorphism groups acting on surfaces. A natural finite graph associated to a surface is its {\it modular flip graph} where one considers triangulations up to homeomorphism. This graph (when defined properly and up to a few exceptions) is exactly the quotient of the flip graph by its graph automorphisms (\cite{KP2} and \cite{Disarlo}). 

In this article we consider a natural variant by measuring distance between triangulations by considering the minimal number of {\it simultaneous} flip moves necessary between them. So in this case, provided flips are made on disjoint quadrilaterals, they can be performed simultaneously. Simultaneous flip distance has been studied in the case of plane triangulations \cite{Bose} (note there is slight difference in the definition of  a triangulation) but also finds its roots in related problems in Teichm\"uller theory. A related problem in surface theory is to measure distance between surfaces - and when these surfaces are hyperbolic and have the same topology, these distances and the related metric spaces give rise to the geometric study of Teichm\"uller and moduli spaces. In these spaces, several of the important metrics (namely the Teichm\"uller metric and the Thurston metric) are $\ell^{\infty}$ metrics. The simultaneous flip metric can be thought of as a combinatorial analogue to these metrics. 

Our main goal is to study the diameters of modular flip graphs of finite type orientable surfaces endowed with this distance. In particular we are interested in how these diameters grow in function of the number of punctures and the genus of the underlying surface. There are two possible types of punctured depending on whether we label the punctures or not. This is equivalent to asking whether we consider homeomorphisms on surfaces that permute the punctures. Our methods allow us to show the following.

\begin{theorem}\label{thm:intro}
There exists a constant $U>0$ such that the following holds. Let $\Sigma_{g,n}$ be a surface of genus $g$ with $n$ labelled marked points. Then any two triangulations of $\Sigma_{g,n}$ are related by at most 
$$
U \left( \log(g+n) \right)^2
$$
simultaneous flip moves.
\end{theorem}

In other terms, the above quantity is an upper bound on the diameter $\diam (\MFS(\Sigma_{g,n}))$ of the modular flip graph. We prove the above result in different contexts and with different explicit constants in front of the leading term; although the constants are explicit, we insist on the fact that it is really the order of growth that we've focussed on. 

We point out that we don't know whether the growth rate is optimal; the best lower bounds we know are on the order of $\log(\kappa)$ in terms of either genus or labelled marked points. It does not seem a priori obvious how to fill the gap nor even what to conjecture might be the right rate of growth (see Section \ref{sec:lower}). However in the case of unlabelled marked points, we show that the growth is at most $\log(n)$ in terms of the number of punctures. As in the case of triangulations of planar configurations of points, it is easy to see that one cannot hope for better (see Section \ref{sec:lower}).

{\bf Organization.}

In the next section we introduce the objects we'll be working with and prove two lemmas we'll use throughout the paper. In Section \ref{sec:spheres} we prove the main theorem for punctured spheres and in Section \ref{sec:genus} for genus $g$ surfaces with a single puncture. These results allow us to deduce the general upper bound in Section \ref{sec:hybrid}. In the final section we discuss lower bounds and further questions.\\

{\bf Acknowledgements.}

The second author is very grateful to the mathematics department of Indiana University for its hospitality during a very nice research visit where parts of this article were written. Both authors acknowledge support from U.S. National Science Foundation grants DMS 1107452, 1107263, 1107367 ÒRNMS: Geometric structures And Representation varietiesÓ (the GEAR Network).

\section{Preliminaries}

In our setup $\Sigma$ is a topological orientable connected finite type surface with a  finite set of marked points on it. Its boundary can consist of marked points and possibly
boundary curves, with the additional condition that each boundary curve has at least one marked point on it. Marked points can be \emph{labelled} or \emph{unlabelled}. Sometimes we will call \emph{punctures} the marked points that do not lie on a boundary curve. We will be interested in the combinatorics of arcs and triangulations of $\Sigma$. The arcs we consider are isotopy classes of simple arcs based at the marked points of $\Sigma$. A \emph{multiarc} is a union of distinct isotopy classes of arcs disjoint except for possibly in their endpoints. A \emph{triangulation} of $\Sigma$ is a maximal multiarc on $\Sigma$ (note that this definition is not standard everywhere). The triangulations we consider here are allowed to contain loops, multiple edges; in particular triangles may share more than a single vertex or a boundary arc. 

We denote by $\kappa(\Sigma)$ the number of arcs in (any) triangulation of $\Sigma$. The Euler characteristic tells us that $\kappa(\Sigma) = 6g +3b +3s +p-6$ where $g$ is the genus of $\Sigma$, $s$ is the number of punctures, $b$ is the number of boundary curves and $p$ is the number of marked points on the boundary curves. 
 
The \emph{modular flip graph} $\MF(\Sigma)$ is a graph whose vertices are triangulations of $\Sigma$ with vertices in the set of marked points of $\Sigma$ up to homeomorphism. The homeomorphisms we consider here preserve the set of marked points; in particular they fix the set of the labelled marked points pointwise and they are allowed to permute the unlabelled marked points. Two vertices of $\MF(\Sigma)$ are joined by an edge if the two underlying triangulations differ by exactly one arc; equivalently two triangulations are joined by an edge if they differ by a \emph{flip}, i.e. the operation of replacing one diagonal with the other one in a square.

\begin{figure}[h]
\leavevmode \SetLabels
\endSetLabels
\begin{center}
\AffixLabels{\centerline{\epsfig{file =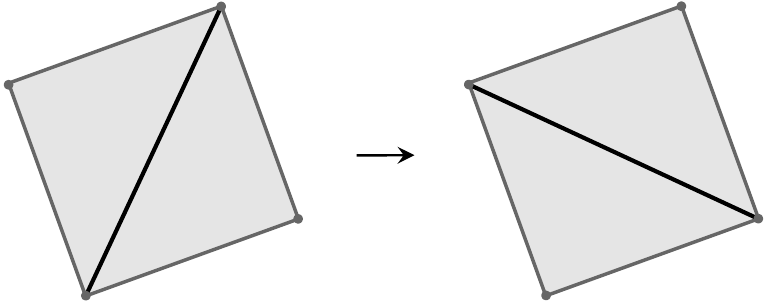,width=6.0cm,angle=0} }}
\vspace{-20pt}
\end{center}
\caption{A flip} \label{fig:flip}
\end{figure}

 An arc that can be flipped is called {\it flippable} and all arcs are flippable except those contained in a punctured monogon (see Figure \ref{fig:unflip})
 
\begin{figure}[h]
\leavevmode \SetLabels
\endSetLabels
\begin{center}
\AffixLabels{\centerline{\epsfig{file =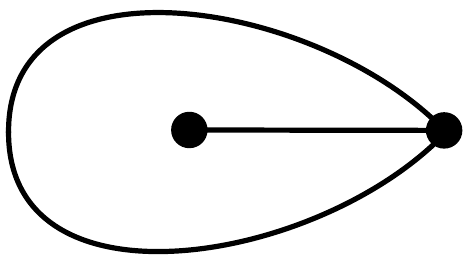,width=5.0cm,angle=0} }}
\vspace{-10pt}
\end{center}
\caption{The central arc is not flippable} \label{fig:unflip}
\end{figure}

 The modular flip graph $\MF(\Sigma)$ can also be described as the quotient of the flip graph of $\Sigma$ modulo the action of the mapping class group (see \cite{Disarlo-Parlier, KP2}). 

In this paper we will be interested in the \emph{modular simultaneous flip graph} $\MFS(\Sigma)$. This is also a graph whose vertices are the triangulations of $\Sigma$ up to homeomorphisms. Here two vertices are joined by an edge if the two underlying triangulations differ by a finite number of flips which are supported on disjoint quadrilaterals on $\Sigma$, i.e. a finite number of flips that can be performed simultaneously on $\Sigma$.

The following result, due to Bose, Czyzowicz, Gao, Morin and Wood, is Theorem 4.4 in \cite{Bose}. It is both a prototype for what we'll be exploring and a tool that we shall exploit. 

\begin{theorem}\label{lem:degreeincreasepolygon}
There exists a constant $K>0$ such that the following is true. Let $P_n$ be a polygon with $n$ vertices and $T,T'$ two triangulations of $P_n$. Then it is possible to relate $T$ to $T'$ in at most $K \log(n)$ simultaneous flips.
\end{theorem}

The constant $K$ is computable and in \cite{Bose} it is shown that $K$ can be taken less than $44$. In the sequel we won't be particularly concerned in optimizing constants as its the order of growth that we're really concerned with. However they all will be computable and we'll indicate exact upper bounds that follow from our methods.

An obvious consequence of the theorem stated above is the following. Given $T$ a triangulation of $P$, let $T_v$ be the unique triangulation of $P_n$ with maximal degree in $v$. Then the simultaneous flip distance between $T$ and $T_v$ is at most $K \log(n)$.   A result of this type is true in any context as stated in the following lemma. 
\begin{lemma}\label{lem:degreeincrease}
Let $v$ be a puncture on a surface $\Sigma$ and $T$ a triangulation of $\Sigma$. Then there exists a sequence of at most $H \log(\kappa(\Sigma))$ simultaneous flips such that the degree of $v$ is maximal. The constant $H$ can be taken equal to $100$. 
\end{lemma}

\begin{proof}

When $\Sigma$ is a polygon this is a consequence of the previous theorem (with a better constant). We can thus suppose that $\Sigma$ has some topology.

We begin by cutting $\Sigma$ along a multiarc made of $2g +n -1$ arcs of $T$ such that the resulting surface is a connected polygon with $4g +2n -2$ sides. 

We now choose a copy of $v_0$ and apply Lemma \ref{lem:degreeincreasepolygon} to increase the degree until it's maximal within the polygon. This step requires at most $K\log(4g + 2n - 2)$ flips.

We now return to the full surface - note that every triangle now has $v_0$ as a vertex. With one simultaneous flip move we can ensure that every triangle has $v_0$ as two of its vertices. To do this consider a triangle with only one copy of $v_0$ as a vertex: exactly one of its three arcs does not have $v_0$ as an endpoint. This arc is flippable, otherwise it surrounds a monogon as in Figure \ref{fig:unflip} and thus there is a triangle without $v_0$ as any of its vertices. As such the triangles with the property of having an arc without $v_0$ as an endpoint come in pairs and form quadrilaterals together. These arcs can all be flipped simultaneously. 

Now it is not difficult to see that with a final simultaneous flip move we can ensure that all triangles have {\it only} $v_0$ as vertices or are what we'll call {\it petals} based in $v_0$. A petal is a triangle like in Figure \ref{fig:unflip} and its base is the exterior vertex. We thus have reached a desired triangulation as the degree is maximal in $v_0$. 

We can now quantify the procedure: the number of simultaneous flip moves is bounded above by
$$K\log(4g+2n-2)+2$$
Finally note that when $\kappa(\Sigma) \geq 2$ we have 
$$ K\log(4g+2n -2) + 2 \leq 100 \log (\kappa(\Sigma)) $$ and this completes the proof.  
\end{proof}

We recall that the \emph{intersection number} $i(a,b)$ between two arcs $a$ and $b$ is defined to be the minimum number of intersection points between two arcs in the classes of $a$ and $b$. 
The intersection number of two multiarcs $A, B$ is  defined as 
$$i(A,B) = \sum_{b \in B} \sum_{a \in A} i(a,b).$$ 
\begin{lemma}\label{lem:arcintroduction}
Let $a$ be an arc and $T$ a triangulation of $\Sigma$ such that $i(a,b)\leq 1$ for all $b\in T$. Then $T$ can be moved in at most $L \log(i(a,T)+1)$ simultaneous flips to a triangulation containing $a$, where $L$ can be taken equal to $100$.
\end{lemma}

\begin{proof}
Assume $i(a, T) \geq 1$. 
Consider the set of all triangles of $T$ through which $a$ passes. They can be assembled into a polygon $P$ and because $a$ only intersects a triangle once, $a$ is a diagonal of this polygon. The polygon has complexity $\kappa = i(a,T)$ by construction so has $i(a,T) + 3$ vertices.
Consider any triangulation $T_a$ of $P$ containing $a$: we now apply Lemma \ref{lem:degreeincreasepolygon} to pass from $T$ to $T_a$ in at most $K \log(i(a,T) + 3) < 100 \log(i(a,T) + 1)$ moves.
\end{proof}

\section{Punctured spheres}\label{sec:spheres}

In this section we focus our attention on finding upper bounds on simultaneous distance between triangulations of punctured spheres and disks with a single marked point on the boundary. 

We begin by proving the following theorem for $\Omega'_n$, a punctured disk with $n$ marked points inside and a single marked point on the boundary. 

\begin{theorem}\label{thm:diskupper}
There exists $A>0$ such that 
$\diam(\MFS(\Omega'_n)) < A (\log(n+1))^2$. The constant $A$ can be taken equal to $1000$. 
\end{theorem}

\begin{proof}
Consider $T,T' \in \MFS(\Omega'_n)$ and denote $v_0$ the boundary vertex of $\Omega_n'$. 

We begin by flipping both $T$ and $T'$ until the degree of $v_0$ is maximal. By Lemma \ref{lem:degreeincrease} this step requires at most $H \log(\kappa(\Omega_n')) =  H \log(3n-2)$  moves for each triangulations. 

The result is a triangulation in which every puncture has an arc joining it to $v_0$ which in turn is surrounded by an arc. As previously, we call the unique triangle containing a given puncture a {\it petal} and the complement of the union of the petals is an $n+1$-gon with $n+1$ copies of $v_0$ as its vertices. 

\begin{figure}[h]
\leavevmode \SetLabels
\endSetLabels
\begin{center}
\AffixLabels{\centerline{\epsfig{file =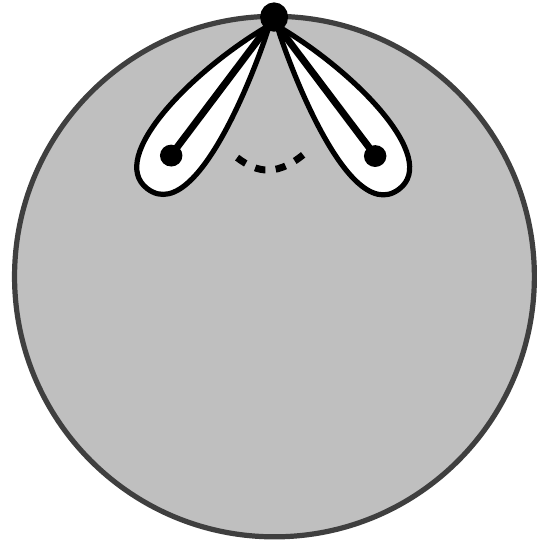,width=4.0cm,angle=0} }}
\vspace{-30pt}
\end{center}
\caption{The shaded area is triangulated (so arcs have both endpoints in $v_0$)} \label{fig:petals}
\end{figure}

For each of our two triangulations we'll now perform the same procedure. We begin by looking at the polygon - one of the edges corresponds to the boundary arc of $\Omega_n'$, say $a$. We give the vertices of the polygon a cyclic order with $p_0$ being on the left of $a$, and $p_{n}$ on the right. 

By Lemma \ref{lem:degreeincreasepolygon} any two triangulations of the polygon are at distance roughly $\log(n)$ apart and we'll use that to obtain a special type of triangulation. 
More precisely we move until the degree of $p_n$ is maximal. By Lemma \ref{lem:degreeincreasepolygon} this step takes at most $K \log(n+1)$ flips. 

\begin{figure}[h]
\leavevmode \SetLabels
\endSetLabels
\begin{center}
\AffixLabels{\centerline{\epsfig{file =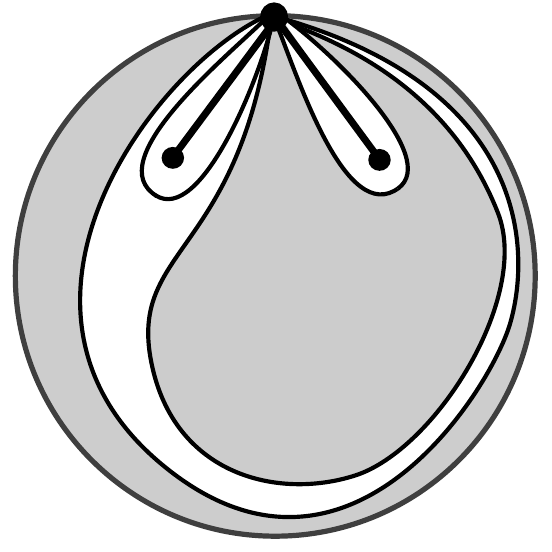,width=4.0cm,angle=0} }}
\vspace{-30pt}
\end{center}
\caption{The shaded areas are triangulated in the same fashion around each petal.} \label{fig:step1}
\end{figure}

We now return to the petals. Figure \ref{fig:step1} represents the result of the previous step around a petal. The goal is split the vertices into two groups, both surrounded by an arc: one with all vertices $v_1$ to $v_{\lfloor \frac{n}{2} \rfloor}$ and one group with the other ones. This can be done in two steps:

The first step takes two moves: flip (simultaneously) all arcs surrounding the petals containing vertices $v_1$ to $v_{\lfloor \frac{n}{2} \rfloor}$ and then flip all arcs between $v_0$ and $v_k$ for $k = 1,\hdots,\lfloor \frac{n}{2} \rfloor$. The result around an individual petal is illustrated in Figure \ref{fig:step2}. 

\begin{figure}[h]
\leavevmode \SetLabels
\endSetLabels
\begin{center}
\AffixLabels{\centerline{\epsfig{file =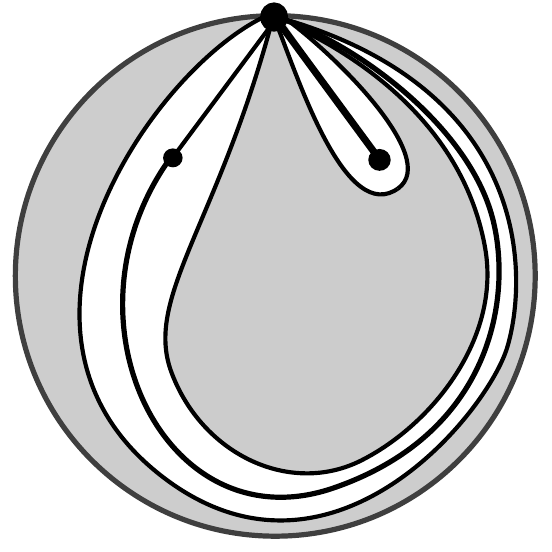,width=4.0cm,angle=0} }}
\vspace{-30pt}
\end{center}
\caption{Breaking the petal...} \label{fig:step2}
\end{figure}

We then flip symmetrically as in Figure \ref{fig:step3}. 

\begin{figure}[h]
\leavevmode \SetLabels
\endSetLabels
\begin{center}
\AffixLabels{\centerline{\epsfig{file =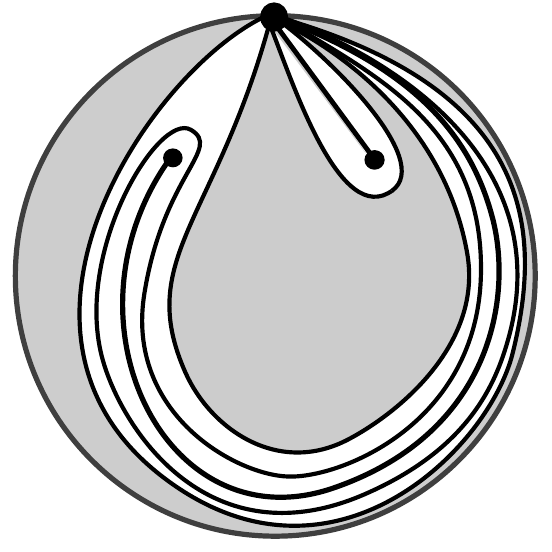,width=4.0cm,angle=0} }}
\vspace{-30pt}
\end{center}
\caption{...and building it else where.} \label{fig:step3}
\end{figure}

The result is again a triangulation with petals but this time the petals with vertices $v_1$ to $v_{\lfloor \frac{n}{2} \rfloor}$ are grouped together with respect to the left-right order.

The second step is to move in the polygon on complement of the petals to create a triangulation which contains two special arcs $b$, $c$: one that surrounds the petals containing $v_1$ to $v_{\lfloor \frac{n}{2} \rfloor}$ and the other that surrounds the remaining petals. Note that $a,b,c$ are the arcs of a triangle. What the rest of the triangulation looks like is irrelevant. By Lemma \ref{lem:degreeincreasepolygon} this step takes at most $K \log(n+1)$ flips.

Now we move (simultaneously) inside each arc $b$ and $c$ which surround resp. $\lfloor \frac{n}{2} \rfloor$ vertices and $n- \lfloor \frac{n}{2} \rfloor$ vertices. Denote by $\Omega_b, \Omega_c$ the two subsurfaces bounded by $b$ and $c$. 
By induction on $n$, the number of flips inside each of the two subsurfaces is at most 
$$ A \log^2(\lfloor \frac{n}{2} \rfloor+1) .$$  
The distance between $T$ and $T' $ is at most 
\begin{align*}
d(T, T') &\leq  A\log^2(\lfloor \frac{n}{2} \rfloor+1) + 2( 2 K\log(n+1) + H \log(3n-2))  + 4 \\ 
            & \leq A \log^2(n+1).
\end{align*}
A direct computation proves that when $A$ is large enough (for example $A =1000$) the last inequality holds for every $n \geq 1$. 
\end{proof}

From the theorem above it is easy to obtain the same type of result for a punctured sphere. 

\begin{theorem}\label{thm:sphereupper}
Let $\Omega_n$ be a sphere with $n$ labelled punctures. Then there exists $B>0$ such that
$\diam(\MFS(\Omega_n)) < B (\log(n))^2$, where $B$ can be taken to be equal to $1100$.
\end{theorem}

\begin{proof}
For $n\leq 3$ the result is immediate since $\Omega_n$ has at most 6 triangulations.  We will now assume $n \geq 4$. 

We could prove the theorem analogously to the previous theorem but for simplicity we'll use the previous result directly. 

Let's denote $v_0, \hdots, v_{n-1}$ the punctures of $\Omega_n$. Given two triangulations $T,T'$ we begin by flipping them to increase the valency of $v_0$ until it is maximal. By Lemma \ref{lem:degreeincrease} this step takes at most $2H\log(\kappa(\Omega_n)) = 2H \log(3n-2)$ moves.
As a result we obtain two triangulations, say $\tilde{T}, \tilde{T}'$, with all vertices with an unflippable arc joining it to $v_0$. Consider the petal surrounding $v_{n-1}$ - the complementary region to it is a triangulation of a disk with a single marked vertex (namely $v_0$) on its boundary and with $n-2$ interior vertices. Theorem \ref{thm:diskupper} tells us that $\tilde{T}$ and $\tilde{T}'$ are at most $A\log^2(n-1)$ apart. We thus have
\begin{align*}
d(T, T')  & \leq d(T, \tilde{T}) + d(T', \tilde{T}') + d(\tilde{T}, \tilde{T}') \\
& \leq 2H \log(3n-2) + A \log^2(n-1) \\
 & \leq  200 \log(3n-2)+1000 \log^2(n-1) \\ 
          &\leq B \log^2(n)
\end{align*}
A direct computation proves that when $B$ is large enough (for example any $B \geq 1100$ works) the last inequality holds for every $n\geq 4$.
\end{proof}

\begin{remark}
The case where the punctures of $\Omega_n'$ are unlabeled is easier. Consider $T, S$ in $\MFS(\Omega'_n)$ and denote $v_0$ the boundary vertex of $\Omega_n'$. We begin by flipping to increase the valence of $v_0$ until it is maximal. By Lemma \ref{lem:degreeincrease} this step requires at most $H \log(\kappa(\Omega_n'))= H \log(3n-2)$. Now up to homeomorphism the two triangulations differ only in a $n+1$-gon (the shaded area of figure \ref{fig:petals}). By Lemma \ref{lem:degreeincreasepolygon} the triangulations $T$ and $S$ differ by at most $$2 H \log(3n-2) + K \log(n+1) <  400 \log(n)$$ simultaneous flips. We have thus proved the following:
\begin{theorem}
Let $\Omega_n'$ be a disk with $n$ unlabelled punctures. 
There exists $B>0$ such that 
$\diam(\MFS(\Omega'_n)) < A\log(n)$, where $A$ can be taken equal to  $400$. 
\end{theorem}
\end{remark}

\begin{remark}
The above proof applies word-by-word for unlabeled punctured spheres $\Omega_n$. We thus have the following: 
\begin{theorem}
Let $\Omega_n$ be a sphere with $n$ unlabelled punctures. There exists $B>0$ such that 
$\diam(\MFS(\Omega_n)) < B \log(n)$, where $B$ can be taken equal to  $400$. 

\end{theorem}
\end{remark}

\section{Surfaces with genus}\label{sec:genus}

In this section we prove our upper bounds in terms of genus. 

For technical reasons we begin by proving a theorem for surfaces of genus $g$ with a single boundary component with a marked point on it.

\begin{theorem}\label{thm:upperboundarygenus}
Let $\Gamma'_g$ be a surface of genus $g$ with a single boundary component with a marked point on it. Then
$$
\diam (\MFS(\Gamma'_g)) < C  \left(\log(g+1) \right)^2
$$
where $C$ can be taken equal to $3000$.
\end{theorem}

We use a technique introduced in Disarlo-Parlier \cite{Disarlo-Parlier} and before proceeding to the proof, we state two lemmas we will need. Proofs can be found in Disarlo-Parlier \cite{Disarlo-Parlier} (Lemmas 4.4 and 4.5). 

\begin{lemma}\label{lem:uppergenus1}
Let $T$ be a triangulation of $\Lambda$, a genus $g\geq 1$ surface with a single boundary curve and $k$ marked points all on the boundary. Then there exists $a \in T$ such that $\Lambda \setminus a$ is connected and of genus $g-1$. 
\end{lemma}

\begin{lemma}\label{lem:uppergenus2}
Let $T$ be a triangulation of $\Lambda$, a genus $g\geq 0$ surface with two boundary curves, both with marked points, and all marked points on the boundary. Then there exists $a \in T$ such that $\Lambda \setminus a$ has only one boundary component.
\end{lemma}

We can now proceed to the proof of the theorem.
\begin{proof}
The result can be checked directly for $g=1$. We need to check that the diameter is at most $2000 \log(2) > 5$. Indeed, a one-holed torus has  at most 5 possible triangulations so the result is true. 

Now suppose that $g \geq 2$. 

Denote by $v$ the boundary vertex of $\Gamma_g'$. Given triangulations $S,T$  of $\Gamma'_g$, flip both until the valence of $v$ is maximal and denote by $S_v, T_v$ the 
triangulations obtained. By Lemma \ref{lem:degreeincrease} each step takes at most $H \log(\kappa(\Gamma_g'))$ flips. Now we proceed as in the proof of Theorem 4.3 in Disarlo-Parlier \cite{Disarlo-Parlier}. 
We successively apply the previous lemmas to find a collection of $2 \lfloor \frac{g}{2} \rfloor$ arcs along which we can cut so that the resulting surface has genus $g-\lfloor \frac{g}
{2} \rfloor$ and a single boundary component with $1+ 4 \lfloor \frac{g}{2} \rfloor$ arcs. (Note that so far we haven't applied a single flip to either $S_{v}$ or $T_{v}$.)

Our aim is now to introduce two special arcs that are essentially parallel to the single boundary of the surface we've obtained by cutting along the arcs (see Figure \ref{fig:genusarc}).

\begin{figure}[h]
\leavevmode \SetLabels
\endSetLabels
\begin{center}
\AffixLabels{\centerline{\epsfig{file =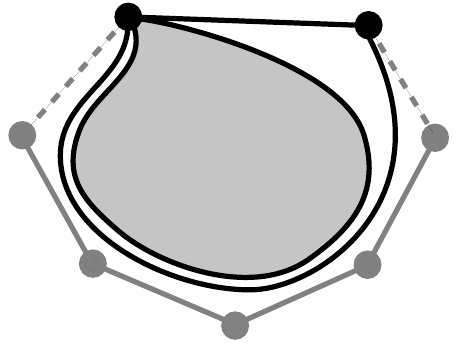,width=4.0cm,angle=0} }}
\vspace{-30pt}
\end{center}
\caption{The shaded region is of genus $g - \lfloor \frac{g}
{2} \rfloor$ } \label{fig:genusarc}
\end{figure}

We describe the process we'll apply to both triangulations $S_v,T_v$. If we consider the arc $a$ which is boundary arc of $\Gamma'_g$, note that the arcs $b,c$ we want to introduce 
form a triangle with $a$ and both cut off (of $\Gamma'_g$) resp. a surface $\Gamma_1'$ of genus $\lfloor \frac{g}{2} \rfloor$, resp. a surface $\Gamma_2'$ of genus $g- \lfloor \frac{g}{2} \rfloor$. 
They also have the nice property of intersecting any arc of the triangulation (either $S_v$ or $T_v$) at most once. In addition this means they intersect the full triangulation at most 
$\kappa(\Gamma_g')$ times. We can now appeal to Lemma \ref{lem:arcintroduction} from the preliminaries which tells us that we can introduce each of them in at most $L \log(\kappa(\Gamma_g')+1)= L \log(6g-1)$ moves. Denote by $S', T'$ the new triangulations obtained, they both contain the arcs $b,c$. Denote by $S_k', T_k'$ the restrictions of $S'$, reps. $T'$ to $\Gamma_k'$ for $k=1,2$. Now flip $S_k'$, $T_k'$ inside $\Gamma_k'$ for $k=1, 2$. Once the triangulations coincide on both $\Gamma_1'$ and $\Gamma_2'$, they will coincide on $\Gamma_g'$.    

By induction on $g$, the following holds:  
\begin{align*}
d(S_1',T_1') &\leq  C \log^2(\lfloor \frac{g}{2} \rfloor +1) \leq C \log^2( \frac{g}{2} +1) \\ 
d(S_2',T_2') &\leq  C \log^2(g- \lfloor \frac{g}{2} \rfloor +1) \leq C \log^2(\frac{g+1}{2} +1) \\
\end{align*}
Putting all together:
\begin{align*}
d(S,T) & \leq d(S,S') + d(T, T') + \max \{d(S_1', T_1') , d(S_2', T_2')\} \\
            & \leq 2 (L \log(6g-1) +H \log(6g-2)) + C \log^2(\frac{g+1}{2}  +1)  \\ 
            & \leq C \log^2(g +1) 
\end{align*}
A direct computation proves that the last inequality holds for every $g \geq 2$ when $C$ is large enough (for instance $C=3000$). 
\end{proof}

We can use the previous theorem to show the analogous result for a genus $g$ surface with a single puncture. 

\begin{theorem}\label{thm:uppergenus}
Let $\Gamma_g$ be a surface of genus $g$ with a single marked point. Then
$$
\diam (\MFS(\Gamma_g)) < C \left( \log(g+1) \right)^2
$$
The constant $C$ can be taken to be equal to $3000$.
\end{theorem}

\begin{proof}
We argue as in the previous theorem by considering for any triangulation a collection of $2 \lfloor \frac{g}{2} \rfloor $ arcs that when cut along give a surface of genus $g- \lfloor \frac{g}{2} \rfloor$ with a single boundary component with $4 \lfloor \frac{g}{2} \rfloor$ arcs. 
As in the above proof, we can introduce an arc that separates the surface in two subsurfaces of genus $\lfloor \frac{g}{2} \rfloor$ and $g - \lfloor \frac{g}{2} \rfloor$. Then we apply the previous theorem to both to obtain the result.
\end{proof}
\section{Hybrid surfaces}\label{sec:hybrid}

In this section we prove our most general upper bound which works for surfaces with punctures and genus. 

\begin{theorem}\label{thm:upperhybrid}
Let $\Sigma_{g,n}$ be a surface of genus $g$ with $n$ labelled marked points. Then
$$
\diam (\MFS(\Sigma_{g,n})) < D \left( \log(g+n) \right)^2
$$
The constant $D$ can be taken equal to $4500$. 
\end{theorem}

\begin{proof}
Consider a triangulation of $\Sigma:=\Sigma_{g,n}$ and a spanning tree of its $1$-skeleton. Note that a spanning tree contains exactly $n-1$ arcs. Consider a marked vertex $v_0$ and the loop $a$ based in 
$v_0$ obtained by leaving from $v_0$ and following the spanning tree along an arc (leaving the spanning tree to the left say) and going around the entire tree before returning to $v_0$. 

\begin{figure}[h]
\leavevmode \SetLabels
\L(.4*.92) $a$\\
\L(.46*.38) $v_0$\\
\endSetLabels
\begin{center}
\AffixLabels{\centerline{\epsfig{file =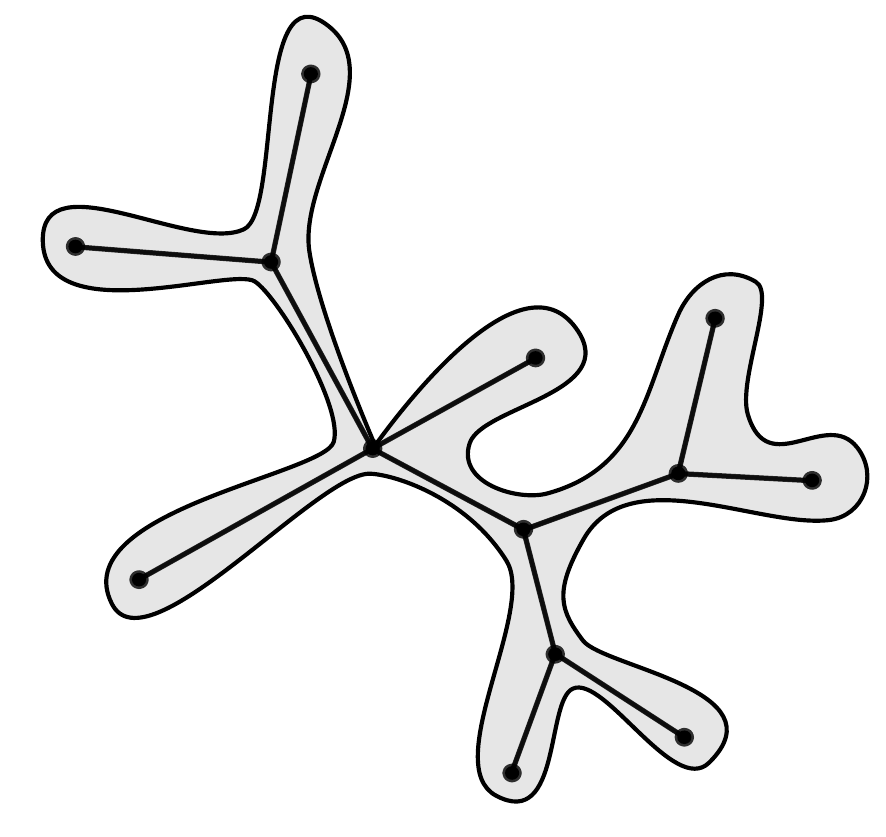,width=6.0cm,angle=0} }}
\vspace{-30pt}
\end{center}
\caption{A spanning tree of the vertices and the arc $a$}\label{fig:spantree}
\end{figure}

The arc $a$ is separating and leaves the genus to one side and the punctures to other (except for the point $v_0$ which lies on the arc itself). We claim that it can be introduced in the triangulation in at most $(H+L) \log(\kappa(\Sigma))$ moves.  

To do so one can proceed as follows. Cutting along the arcs of the spanning tree we find a surface $\Sigma^{\beta}$ of genus $g$ and a single polygonal boundary component $\beta$ with all marked points now on the boundary. A marked point of degree $d$ in the spanning tree appears on the boundary component $\beta$ exactly $d$ times and $\beta$ is a polygon of $2n-2$ arcs (twice the number of arcs of the spanning tree).

Note that the arc $a$ also lives on $\Sigma^{\beta}$ and is a loop parallel to $\beta$ with its basepoint a copy of $v_0$. We now flip the restriction of the triangulation to increase the valence of the basepoint of $a$ until it is maximal. By Lemma \ref{lem:degreeincrease} this step requires at most $H\log(\kappa(\Sigma) - (n-1)) < H\log(\kappa(\Sigma)) $ simultaneous flips. The arc $a$ now intersects any arc in the triangulation at most once and thus by Lemma \ref{lem:arcintroduction} can be introduced in at most $L \log(\kappa(\Sigma))$ moves. 

This can be done to any triangulation so now considering two triangulations $T$ and $S$, we perform the above process on both. The new triangulations obtained, say $S'$ and $T'$, possibly differ in ``the genus part" $\Gamma_g'$ or the ``puncture part" $\Omega_{n-1}'$ but by applying Theorems \ref{thm:diskupper}, \ref{thm:upperboundarygenus}  from before, we can conclude that they lie at distance at most 
\begin{align*}
d(S,T) & \leq d(S, S') + d(T, T') + d(S', T') \\  
& \leq 2  (H+L) \log(\kappa(\Sigma)) + \max \{\diam (\MFS(\Gamma_g')) , \diam (\MFS(\Omega_{n-1}'))\} \\ 
& \leq 2 (H+L) \log(6g+3n-6) + C \log^2(g+n) \\ 
& \leq D \log^2(g+n).
\end{align*}
A direct computation proves that the last inequality holds for every $g, n$ such that $g+n\geq 2$ provided that $D$ is large enough (for example, $D=4500$).
\end{proof}

\begin{remark} We can apply the same proof as above to the case of a surface with unlabelled marked points. One has to be careful because in the above estimates, we are trying to capture the cases where both the genus and number of points are increasing, possibly at different rates. Our previous upper bounds for spheres with unlabelled marked points grows $\log(n)$; in combination with the above proof this implies that for fixed genus, one can again obtain an upper bound on the order of $\log(n)$ with an additive constant that depends on the genus. Again, all constants can be made explicit but for simplicity we won't discuss this in detail. 
\end{remark}

\section{Lower bounds and further questions}\label{sec:lower}

For surfaces with genus and labelled marked points, our upper bounds grow roughly like $(\log(\kappa))^2$ in $\kappa$ the complexity of the surface. It is not clear where this order of growth is optimal.

An immediate lower bound can be deduced from known bounds on the diameters of the usual flip graphs. In those cases lower (and upper) bounds are known to grow like $g \log(g) + n \log(n)$ (see Theorem 1.4 and Corollary 4.19 of \cite{Disarlo-Parlier}). As at most a linear number of flips in terms of the complexity can be performed simultaneously, this implies a lower bound on the order of $\log(\kappa)$. The counting argument used to provide this bound is pretty simple, especially compared to our upper bounds, and it does not seem particularly adapted for simultaneous flips. In terms of unlabelled marked points, the {\it same} order of growth holds for these lower bounds. It would seem surprising that there is no difference in order of growth between labelled and unlabelled marked points. All of these points seem to indicate that perhaps a better lower bound might be achievable.

On the other hand, there are some indications that an upper bound on the order of $\log(\kappa)$ might be possible. A seemingly related problem to estimating distances in the flip graph is the problem of estimating distances between $3$-regular graphs using Whitehead moves. These graphs are dual to a triangulation and a flip on a triangulation corresponds to a Whitehead move. Triangulations are  really different though; first of all they really correspond to ribbon graphs and not $3$-regular graphs. Secondly only certain Whitehead moves on a $3$-regular graph can be emulated by flips. In particular, it's not possible to deduce results about flip distances from estimates on Whitehead moves or vice-versa. But although the relationship is not direct, there have been a number of recent results that seem to indicate similar behaviors. The $\kappa \log(\kappa)$ behavior discussed previously for modular flip graphs is also present for Whitehead moves on graphs (see for instance \cite{Cavendish, Cavendish-Parlier}). Simultaneous flip moves are thus related to simultaneous Whitehead moves and Rafi and Tao have shown that the growth for graphs behaves like $\log(\kappa)$. This seems to indicate that perhaps our upper bounds might be improvable. A further indication that this order of growth might be correct are the results in \cite{Bose} that were among the tools needed for our upper bounds.

In short, we now know that the rough behavior in terms of either the genus or number of labelled marked points is bounded below and above by a function of type $\log(\kappa)^\alpha$ for $\alpha \in [1,2]$ and determining the exact behavior might be an interesting problem.

\addcontentsline{toc}{section}{References}
\bibliographystyle{amsplain}
\bibliography{Simultaneous.bib}

{\em Addresses:}\\
Department of Mathematics, University of Fribourg, Switzerland\\
{\it and} Hunter College, City University of New York, NY, USA \\
Indiana University, Bloomington IN, USA \\
{\em Emails:} \href{mailto:hugo.parlier@unifr.ch}{hugo.parlier@unifr.ch}, \href{mailto:valentina}{vdisarlo@indiana.edu}\\

\end{document}